\documentclass[12pt]{article}
\usepackage[utf8]{inputenc}
\usepackage{amssymb, amsmath, amsfonts}
\usepackage{amsthm}
\usepackage[a4paper, total={6in, 9in}]{geometry}
\usepackage[parfill]{parskip}

\usepackage{tikz}
\usetikzlibrary{decorations.markings}
\usepackage{comment}
\newtheorem{theorem}{Theorem}
\newtheorem*{MT}{Menger's Theorem}
\newtheorem{lemma}[theorem]{Lemma}
\newtheorem{proposition}[theorem]{Proposition}
\theoremstyle{definition}
\newtheorem*{notation}{Notation}
\newtheorem*{ack}{Acknowledgement}

\newtheoremstyle{note}
{3pt}{3pt}{}{}
{\itshape}
{\normalshape{:}}
{.5em}{}
\theoremstyle{note}
\newtheorem{clm}{Claim}

\newtheoremstyle{pfs}
{}{}{}{1em}
{\itshape}
{\itshape{.}}
{.5em}{}
\theoremstyle{pfs}
\newtheorem*{pf}{Proof}

\DeclareMathOperator{\Gal}{Gal}

\newcommand{\qedtri}{\hfill \ensuremath{\triangle}}

\begin{document}

\title{Improved upper bounds on longest-path and maximal subdivision transversals}
\author{H. A. Kierstead\thanks{Email: kierstead@asu.edu} ~and E. R. Ren\thanks{Email: eren2@asu.edu}\\
 Arizona State University\\ Tempe, AZ, USA}
\maketitle
\begin{abstract}
   Let $G$ be a connected graph on $n$ vertices. The Gallai number $\Gal(G)$ of $G$ is the size of the smallest set of vertices that meets every maximum path in $G$. Grünbaum constructed a graph $G$
   with $\Gal(G)=3$. Very recently,  Long, Milans, and Munaro,
    proved that $\Gal(G)\le 8n^{{3}/{4}}$. This was the first  sublinear upper
   bound on $\Gal(G)$ in terms of $n$. We
   improve their bound to $\Gal(G)\leq
   5 n^{{2}/{3}}$. We also tighten a more general result of 
   Long et al. For a multigraph $M$, we prove that if the set $\mathcal L(M,G)$ of 
   maximum $M$-subdivisions in $G$ is pairwise intersecting  and 
   $n\ge m^{6}$, then $G$ has a set of vertices with size at most
    $5 n^{{2}/{3}}$ that meets every $Q\in \mathcal L(M,G)$
\end{abstract}

\section{Introduction}
For a graph $G=(V,E)$, let $|G|=|V|$ and $\|G\|=|E|$.  Two graphs or vertex sets \emph{meet} if they have a common vertex. Define the Gallai number  $\Gal(G)$ of $G$ to be the size of the smallest set $S\subseteq V$ that meets every longest path in $G$. 
It is folklore \cite[Exercise 1.2.40]{West} that 
if $G$ is connected then any two longest paths have a common vertex. 
This result prompted Gallai  \cite{Gal} to ask
whether $\Gal(G)=1$ for all connected graphs $G$.
Walther \cite{Wal} found a graph $G$ with
$\Gal(G)=2$ and $|G|=25$. Then Walther and  Voss \cite{WV}, and
independently Zamfirescu \cite{Zam2}, 
observed that replacing a vertex $v$ of the Petersen
graph by three leaves, each 
adjacent to a distinct neighbor
of $v$, yields a graph  $G$ with $\Gal(G)=2$ and
$|G|=12$. Grünbaum \cite{Gru} later produced a graph $G$ with
$\Gal(G)=3$ and $|G|= 324$; soon after Zamfirescu
\cite{Zam2} found a 270-vertex example. Both Walther
and Zamfirescu \cite{Zam1} then posed the still-open question of
whether the Gallai number has a constant upper
bound for
connected graphs. It even remains open whether any
connected graph $G$ has $\Gal(G)\geq 4$. On the
upper end, Rautenbach and Sereni  \cite{RS} proved in 2014
that all connected graphs $G$ satisfy $\Gal(G)\leq \lceil
|G|/4-|G|^{{2}/{3}}/90 \rceil$. 
In 2020, Long,
Milans, and Munaro \cite{MML} proved the first sublinear upper bound: 
$\Gal(G)\le 8|G|^{3/4}$. 

In fact, Long et al.\ proved a more general result. 
Let $M$ be a connected multigraph. An $M$-\emph{subdivision} is (a copy of) a graph 
obtained from $M$ by subdividing each of its edges $0$ or more times. 
For example, a path is a $K_{2}$-subdivision, and a cycle is  
a $C_{1}$-subdivision, where $C_{1}$ is the multi-cycle with one vertex and one loop. 
 An $M$-subdivision  $Q\subseteq G$ is \emph{maximum} if no $M$-subdivision has more edges than $Q$. 
Let $\mathcal L(M,G)$ be the set of maximum $M$-subdivisions in $G$. A \emph{transversal}\footnote{We use ``transversal'' in the hypergraph sense, i.e., a vertex-cover, not in the combinatorics sense, which would require unique intersections.} of $\mathcal L(M,G)$ is a set of vertices $S\subseteq V(G)$ that meets every $Q\in \mathcal L(M,G)$.
Let $\tau(M,G)$, be the size of a minimal transversal of $L(M,G)$. 
So $\tau(K_{2},G)=\Gal(G)$. 
A family of sets (e.g.\ $\mathcal L(M,G)$) is \emph{pairwise intersecting} if any two of its members meet.
As shown in \cite{MML}, it is easy to check that if $G$ is $(\|M\|^{2}+1)$-connected then $\mathcal L(M,G)$ is pairwise intersecting. The following proposition somewhat improves this bound when $M$ has cut edges, and it is tight for our two  motivating examples, $M\in \{K_{2},C_{1}\}$. We prove it in Section~3.

\begin{proposition}\label{prop:I}
Let $M$ be a connected multigraph with $c$ cut-edges, and let $G$ be a graph. If $\mathcal L(M,G)$ is not pairwise intersecting then 
the connectivity of $G$ is at most $\|M\|^{2}-c\|M\|+\binom{c}{2}$.
\end{proposition}

 Long et al.\ proved:
\begin{theorem}[\cite{MML}]
Let $M$ be a connected multigraph and $G$ be a graph. If $\mathcal L(M,G)$ is pairwise intersecting then $\tau(M,G)\le 8\|M\|^{5/4}|G|^{3/4}$.
\end{theorem}

Notice that if $|G|\le\|M\|^{5}$ then  
$\tau(M,G)\le|G|< 8\|M\|^{5/4}|G|^{3/4}$. Here we tighten their 
bound  by lowering the exponents when $|G|>\|M\|^{3}$, 
and by eliminating the dependence on $M$ 
when $|G|>\|M\|^{6}$.

\begin{theorem}\label{Main} Let $M$ be a connected multigraph and $G$ be a graph. If $\mathcal L(M,G)$ is pairwise intersecting and $\|M\|^{3}<|G|$ then $\tau(M,G)\le \max\{5|G|^{2/3},2\|M\|^{2}|G|^{1/3}\}$. In particular, if $|G|\ge\|M\|^{6}$ then $\tau(M,G)\le5|G|^{2/3}$.
\end{theorem}

\begin{notation} Fix $i,j,n\in \mathbb N$. 
Set $[n]:=\{i\in\mathbb N:1\le i\le n\}$ and $i\oplus j:=i+j\mod n$. 
We denote a path $P$ with $V(P)=\{v_1,\dots,v_{n}\}$ and
 $E(P)=\{v_1v_{2},\dots, v_{n-1}v_{n}\}$ 
 by $v_1\dots v_n=v_{n}\dots v_{1}$, 
 and set $v_iP:=v_i\dots v_n$, $Pv_i:=v_1\dots v_i$ and 
 $v_iPv_j:=v_i\dots v_j$. We form walks by concatenating paths: in a 
 graph $G$, if $Q:=w_1\dots w_s$ then $Pv_iw_jQ$ is defined to be 
 $v_1\dots v_iw_j\dots w_s$ if $v_iw_j\in E(G)$ and $v_1\dots 
 v_iw_{j+1}\dots w_s$ if $v_i=w_j$; else it is undefined. 
 Let $d_{G}(x,y)$ denote the distance form $x$ to $y$ in $G$. 
 The cycle $C:=P+v_{n}v_{1}$ is denoted by $v_{1}\dots v_{n}v_{1}$. 
 If $Q:=v_{i}\dots v_{j}$ then  $v_{i}(C-E(Q))v_{j}$ is the 
 path  $v_{i}v_{i\oplus (n-1)}...v_{j\oplus 1}v_{j}$. 
 Let $v_{i}Cv_{j}$ denote the longer of $v_{i}(C-v_{j\oplus 1})v_{j}$ 
 and $v_{i}(C-v_{i\oplus1})v_{j}$ with ties going to the first. 
Let $A,B\subseteq V(G)$. An $A,B$-path is a path  
$P=v_{1}\dots v_{t}$ with $V(P)\cap A=\{v_{1}\}$ and
 $V(P)\cap B=\{v_{t}\}$.  
An $A,B$\emph{-separator} is a set $S\subseteq V(G)$ that meets every 
$A,B$-path. An $A,B$\emph{-connector} is a set of 
disjoint $A,B$-paths.  As in 
 Diestel \cite{Die}, we assume that $V(G)\cap E(G)=\emptyset$ and treat 
 $G$ as the set  $V(G)\cup E(G)$.
\end{notation}
We will need Menger's Theorem.
\begin{MT}[\cite{Men}] Let $G$ be a graph with $A,B\subseteq V(G)$.
 If $S$ is a minimum $A,B$-separator in $G$ and $\mathcal T$ is a maximum $A,B$-connector in $G$, then $|S|=|\mathcal{T}|$.
\end{MT}

\section{Small transversals of $\mathcal L(M,G)$}

In this section we prove Theorem~\ref{Main}. First we simplify our notation for this and the next section. Fix a connected multigraph $M=(W,F)$ with 
$\|M\|=:m$ and a graph $G$ with $|G|=:n$. For $H\subseteq G$, let $
\mathcal L(H):=\mathcal L(M,H)$,  $\tau(H):=\tau(M,H)$, $\mathcal 
L:=\mathcal L(G)$ and $\mu:=|Q|$, where $Q\in \mathcal L$. An 
$H$-\emph{transversal} is a transversal of $\mathcal L(H)$.  For a graph $Q\in \mathcal L$ and an
 edge $e:=uv\in F$, let $q^{u},q^{v}, Q_{e}$ be the branch vertices and subdivided edge of $Q$ corresponding to $u,v,e$.
We  start with two (easy) general lemmas on paths and cycles. 

\begin{lemma}\label{conpath}
Let $C=v_{1}P_{1}v_{2}P_{2}v_{3}P_{3}v_{4}P_{4}v_{1}$ be a cycle, where each $P_{i}$ is a path with $\|P_{i}\|\ge1$. Then $\|P_{1}\|<\|P_{2}v_{3}P_{3}v_{4}P_{4}\|$ or $\|P_{3}\|<\|P_{4}v_{1}P_{1}v_{2}P_{3}\|$. \qed
\end{lemma}

\begin{lemma}\label{short}
Suppose $C\subseteq G$ is a cycle and $P=x\dots y\subseteq G$ is a path  with $x,y\in C$. If $\|P\|< d_{C}(x,y)$ then ther
e is a cycle $C'\subseteq G$ with $|C|/2<|C'|<|C|$. 
\end{lemma}

\begin{proof}
Let $V(P\cap C)=\{x_{0},\dots,x_{t}\}$, where   $P=x_{0}P_{1}x_{1}\dots x_{t-1}P_{t}x_{t}$, $x_{0}=x$, and $x_{t}=y$. As
\[\sum_{i\in [t]}\|P_{i}\|=\|P\|<d_{C}(x,y)\le \sum_{i\in[t]}d_{C}(x_{i-1},x_{i}),\]
there is $i\in[t]$ with $\|P_{i}\|<d(x_{i-1},c_{i})\le|C|/2$. Let $C'$ be the cycle obtained by replacing the short  $x_{i-1},x_{i}$-path in $C$ by $P_{i}$. 
\end{proof}

\begin{proof}[Proof of Theorem \ref{Main}]
Assume $\mathcal L$ is pairwise intersecting and $M^{3}<n$. We will show that $\tau(\mathcal L)\le 5 n^{2/3}$. Note that $\tau(\mathcal L)\le\mu$. 

Similarly to  \cite{MML}, a pair $(X, Y)$  is called an $H$-\emph{pretransversal} if  
$Y\subseteq X\subseteq V(H)$, $n^{1/3}|Y|\leq  |X|$, and for every graph 
$Q\in \mathcal{L}(H)$ either $Q\subseteq H-X$ or $Q \cap Y \neq 
\emptyset$. Then $Y$ is a $G$-transversal with $|Y|\le n^{2/3}$  
if and only if $(V,Y)$ is a $G$-pretransversal. 
Let $(X, Y)$ be a $G$-pretransversal with $|X|$ 
maximum; it exists because $(\emptyset, \emptyset)$ is a candidate. If 
$X=V$ then we are done with $\tau(\mathcal L)\le n^{2/3}$.  Else set $H:=G-X$. 
If $(X',Y')$, is an $H$-pretransversal then $(X'\cup X, Y\cup Y')$ is a $G$-pretransversal.
So by maximality:
\begin{equation}\label{Max}
    \mbox{There is no $H$-pretransversal.}
\end{equation}
If $S$ is an $H$-transversal then $S\cup Y$ is a $G$-transversal. 
Thus it suffices to show that
 $\tau(\mathcal L(H))\le \max \{m^{2}n^{1/3}, 4n^{2/3}\}$.  Arguing 
 by contradiction, assume
 \begin{equation}\label{Contra}
    \tau(\mathcal L(H))> \max \{m^{2}n^{1/3}, 4n^{2/3}\}
\end{equation}
 As we are working only in $H$, let $\tau:=\tau(\mathcal L(H))$, and let \emph{transversal} mean $H$-transversal. The proof follows easily from the next four claims. The last two depend on the first, but there are no other dependencies. Some readers may prefer to skip to the last paragraph of this section before reading the claims and their proofs.

\begin{clm}[\cite{MML}]\label{Sep} For all disjoint sets
$A,B\subseteq V(H)$ with $s:=\min(|A|,|B|)$, there is
an $A,B$-connector $\mathcal K$ in $H$  with $|\mathcal K|\geq s/n^{1/3}$.
\end{clm}
\vspace{-.75em}
\begin{pf} Suppose not. By Menger's Theorem,
there is an $A,B$-separator $Y'$ in $H$ such that
$n^{1/3}|Y'|<s$. As $\mathcal L$ is pairwise intersecting, there is a
component  $H'$ of $H- Y'$ such that no
other component of $H-Y'$ contains a member of $\mathcal L(\mathcal H)$. 
Since $Y'$ separates $A$ and $B$, at least one
(say $A$) of $A$, $B$ must be contained in
$X':=V(H- H')$. 
Thus all $Q\in\mathcal L(H)$ intersect $Y'$
or are contained in $H'$.
 Also $n^{1/3}|Y'|\leq s\le |A|\leq
 |X'|$.  So  $(X',Y'$) is an
 $H$-pretransversal, contradicting
 \eqref{Max}. 
  \qedtri
\end{pf}

\begin{clm}\label{base}
$H$ has a cycle $C_{0}$ with $|C_{0}|>mn^{1/3}$.
\end{clm}
\vspace{-.75em}
\begin{pf}
Let $Q\in \mathcal L(H)$, and pick a minimum transversal $S$ subject to $S\subseteq V(Q)$. By \eqref{Contra}, 
$|S|> m^{2}n^{1/3}$, so there is an edge $e\in F$ with 
$|Q_{e}\cap S|>mn^{1/3}$. Pick a shortest path 
$P:=q_{i}\dots q_{j}\subseteq Q_{e}$ such that $(S\smallsetminus V(Q_{e}))\cup V(P)$ is a transversal.
By minimality, $|P|\ge|Q_{e}\cap S|\ge mn^{1/3}$, 
and for each end $v$ of $P$ there is an $M$-subdivision
 $R(v)\in \mathcal L(H)$ that meets $P$ only at $v$. 
As $\mathcal L$ is pairwise intersecting, there is a $q_{i},q_{j}$-path 
$R\subseteq R(q_{i})\cup R(q_{j})$. Now $C_{0}:=P\cup R$ is a cycle 
with $|C_{0}|>mn^{1/3}$. \qedtri
\end{pf}

\begin{clm}\label{nottr}
Suppose $C\subseteq G$ is a cycle that is not a transversal. Then (i) if $|C|> mn^{1/3}$ then $G$ has a cycle that is longer than $C$ and (ii) $|C|\le 2n^{2/3}$.
\end{clm}
\vspace{-.75em}
\begin{pf}
Let $Q\in \mathcal L(H)$ witnesses that 
$C$ is not a transversal. 
By Claim~\ref{Sep}, there is a $C,Q$-connector $\mathcal T$ 
with $|\mathcal T|>\min\{|C|,\mu\}/n^{1/3}$. 

(i)  Suppose $|C|> mn^{1/3}$. Pigeonholing, there is an edge $e\in F$ such that $Q_{e}$ meets $\mathcal T$ at least twice.  Say $T_{i}=x_{i}\dots y_{i}\in \mathcal T, i\in[2]$. Applying Lemma~\ref{conpath}, to the cycle $C':=x_{1}T_{1}y_{1}Q_{e}y_{2}T_{2}x_{2}Cx_{1}$, and using the maximality of $Q$, yields that $|C'|>|C|$.

\begin{figure}[h!]
    \label{fig:Claim3}
    \center{
    \begin{tikzpicture}[scale=0.7]
    \node (C) at (0,0) {};
    \draw[very thick] (C) ellipse [x radius = 10, y radius = 1];
    \node (Q) at (-10,4.5) {};
    \draw[fill,opacity=0.07] (Q) rectangle (10,2.5);
    
    \node[fill,label=$q^w$,shape = circle,inner sep = 0 cm, minimum size = .25 cm] (w) at (-9,3) {};
    \node[fill,label=$q^x$,shape = circle,inner sep = 0 cm, minimum size = .25 cm] (x) at (-5,3) {};
    \node[fill,label=below:$q^y$,shape = circle,inner sep = 0 cm, minimum size = .25 cm] (y) at (1,4) {};
    \node[fill,label=below:$q^z$,shape = circle,inner sep = 0 cm, minimum size = .25 cm] (z) at (9,4) {};

    \node[fill,shape = circle,inner sep = 0 cm, minimum size = .18 cm] (q1) at (-7,3) {};
    \node[fill,shape = circle,inner sep = 0 cm, minimum size = .18 cm] (q2) at (-3,3.33) {};
    \node[fill,shape = circle,inner sep = 0 cm, minimum size = .18 cm] (q3) at (-1,3.66) {};
    \node[fill,shape = circle,inner sep = 0 cm, minimum size = .18 cm] (q4) at (3,4) {};
    \node[fill,shape = circle,inner sep = 0 cm, minimum size = .18 cm] (q5) at (5,4) {};
    \node[fill,shape = circle,inner sep = 0 cm, minimum size = .18 cm] (q6) at (7,4) {};
    
    \node[fill,shape = circle,inner sep = 0 cm, minimum size = .18 cm] (c1) at (-7,0.71) {};
    \node[fill,shape = circle,inner sep = 0 cm, minimum size = .18 cm] (c2) at (-3,0.93) {};
    \node[fill,shape = circle,inner sep = 0 cm, minimum size = .18 cm] (c3) at (-1,1) {};
    \node[fill,shape = circle,inner sep = 0 cm, minimum size = .18 cm] (c4) at (3,0.95) {};
    \node[fill,shape = circle,inner sep = 0 cm, minimum size = .18 cm] (c5) at (5,0.85) {};
    \node[fill,shape = circle,inner sep = 0 cm, minimum size = .18 cm] (c6) at (7,0.72) {};
    
    \node[label={[red]:$\mathcal{T}$}] (t) at (-10.5,1.2) {};
    \node at (-10.5,0) {$C$};
    \node at (-10.5,3.5) {$Q$};

    \draw[very thick, red] (c1)--(q2);
    \draw[very thick, red] (c2)--(q1);
    \draw[very thick, red] (c3)--(q4);
    \draw[very thick, red] (c4)--(q3);
    \draw[very thick, red] (c5)--(q5);
    \draw[very thick, red] (c6)--(q6);

    \draw[thick, black] (w)--(q1)--(x)--(q2);
    \draw[ultra thick, blue] (q2)--(q3);
    \draw[thick, black] (q3)--(y)--(q4);
    \draw[ultra thick, blue] (q4)--(q5)--(q6);
    \draw[thick, black] (q6)--(z);
    
    \node[label={[blue]:$I$}] (I) at (-2, 3.5) {};
    
    \end{tikzpicture}
    }
    \caption{\label{fig1}Partitioned $Q_{wx},Q_{xy},Q_{yz}$ with $0,1,2$ inner paths}
\end{figure}
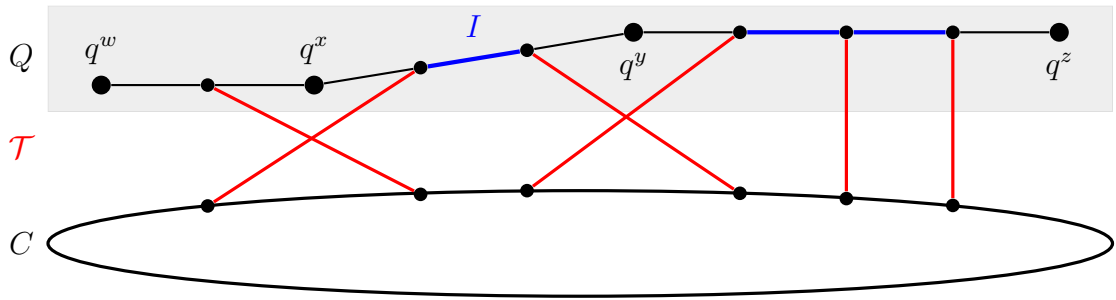

 (ii)  Suppose $|C|> 2n^{2/3}$. For each edge $e=uv\in F$, let $h_{e}$ be 
 the number of paths in $\mathcal T$  that  end in $Q_{e}$.  
  When $h_{e}\ne 0$ these ends 
partition $E(Q_{e})$   into $h_{e}-1$ \emph{inner} paths and two
 \emph{outer} paths containing $q^u$ or $q^v$. See Figure~\ref{fig1}. 
 The number of inner subpaths is at least 
 $$\sum_{e\in F}(h_{e}-1)=|\mathcal T|-m>2n^{1/3}-m>n^{1/3}.$$ 
 By Lemma~\ref{conpath},  $\|I\|>|C|/2> n^{2/3}$ for all inner 
 paths $I$, so $|Q|>n$, a contradiction.~$\triangle$
\end{pf}

\begin{clm}\label{small}
Suppose $C=v_{1}\dots v_{l}v_{1}\subseteq G$ is a cycle with $l> 4n^{2/3}$. Then there is a cycle $C'\subseteq G$ with $l/2<|C'|<l$. 
\end{clm}
\vspace{-.75em}
 \begin{pf}
 Let  $l=4k+r, 0\le r<4$. Define paths  $P_{1}=v_{1}\dots v_{k+1}$, $P_{2}=v_{k+2}\dots v_{2k+1}$, $P_{3}=v_{2k+2}\dots v_{3k+3}$, and $P_{4}=v_{3k+4}\dots v_{l}$, where $a\in\{0,1\}$. Then 
 $$C=v_{1}P_{1}v_{k+1}v_{k+2}P_{2}v_{2k}v_{2k+1}P_{3}
 v_{3k+1}v_{3k+2}
 P_{4}v_{l}v_{1},$$
 $|P_{1}|=k+1=|P_{3}|$ and $k-1= |P_{2}|\le|P_{4}|$.
  By Claim~\ref{Sep}, there is a $P_{1},P_{3}$-connector $\mathcal T$ with $|\mathcal T|\ge(k+1)/n^{1/3}>n^{1/3}$. Pigeonholing, there is a path $T=x\dots y\in \mathcal T$ with 
 $$\|T\|+1\le|T|\le \lfloor\frac{n}{|\mathcal T|}\rfloor\le\lfloor n^{2/3}\rfloor< k+1\le|P_{2}|+2\le d_{C}(x,y)+1.$$ 
 Now $\|T\|<d_C(x,y)$, so  we are done by Lemma~\ref{short}. See Figure~\ref{fig:Claim4}. \qedtri 
 \end{pf}

\begin{figure}[h!]
    \label{fig:Claim4}
    \center{
    \begin{tikzpicture}
    \pgfdeclarelayer{bg}    
    \pgfsetlayers{bg,main}
    
    \node[color=red] (C) at (0,0) {$T$};
    
    \draw[color=black!40] (C) ellipse [x radius = 7, y radius = 2];
    \draw[very thick,color=blue] (3,1.82) arc(65:115:7 and 2);
    \draw[very thick,color=blue] (-3,-1.82) arc(245:295:7 and 2);
    
    \draw[very thick] (-4,1.65) arc(125:235:7 and 2);
    \draw[very thick] (4,-1.65) arc(305:415:7 and 2);
    
    \node[] at (-7.5,0) {$P_1$};
    \node[color=blue] at (0,2.4) {$P_2$};
    \node[] at (6.5,0) {$P_3$};
    \node[color=blue] at (0,-2.4) {$P_4$};
    
    \node[shape = circle,inner sep = 0 cm, minimum size = .25 cm] at (0,2) {};
    \node[fill,shape = circle,inner sep = 0 cm, minimum size = .25 cm] at (1,1.97) {};
    \node[fill,shape = circle,inner sep = 0 cm, minimum size = .25 cm] at (2,1.91) {};
    \node[fill,label=$v_{2k}$,shape = circle,inner sep = 0 cm, minimum size = .25 cm] at (3,1.82) {};
    \node[fill,shape = circle,inner sep = 0 cm, minimum size = .25 cm] (y) at (4,1.65) {};
    \node[fill,shape = circle,inner sep = 0 cm, minimum size = .25 cm] (p3) at (-1,1.97) {};
    \node[fill,label=$v_{k+2}$,shape = circle,inner sep = 0 cm, minimum size = .25 cm] at (-3,1.82) {};%
    \node[fill,shape = circle,inner sep = 0 cm, minimum size = .25 cm] (x) at (-4,1.65) {};
    
    \node[fill,label=below:$v_{l}$,shape = circle,inner sep = 0 cm, minimum size = .25 cm] at (-3,-1.82) {};
    \node[fill,label=below:$v_1$,shape = circle,inner sep = 0 cm, minimum size = .25 cm] at (-4,-1.65) {};
    
    \node[fill,label=below:$v_{3k+2}$,shape = circle,inner sep = 0 cm, minimum size = .25 cm] at (3,-1.82) {};
    \node[fill,label=below:$v_{3k+1}$,shape = circle,inner sep = 0 cm, minimum size = .25 cm] at (4,-1.65) {};
    
    \node at (-4.25,1.9) {$x=v_{k+1}$};
    \node at (4.4,2) {$y=v_{2k+1}$};
    
    \node[fill,shape = circle,inner sep = 0 cm, minimum size = .25 cm] (p2) at (-3,1) {};
    \node[fill,shape = circle,inner sep = 0 cm, minimum size = .25 cm] (p5) at (-0.45,0) {};
    \node[shape = circle,inner sep = 0 cm, minimum size = .25 cm] (p6) at (0,-2) {};
    \node[fill,shape = circle,inner sep = 0 cm, minimum size = .25 cm] (p7) at (1,-1.97) {};
    \begin{pgfonlayer}{bg}
        \draw [very thick,color=red] plot [smooth] coordinates {(x)(p2)(p3)(p5)(p7)(y)};
    \end{pgfonlayer}

    \end{tikzpicture}}
    \caption{$||T||+1<|P_2|+2$ suffices (a tight example).}

\end{figure}
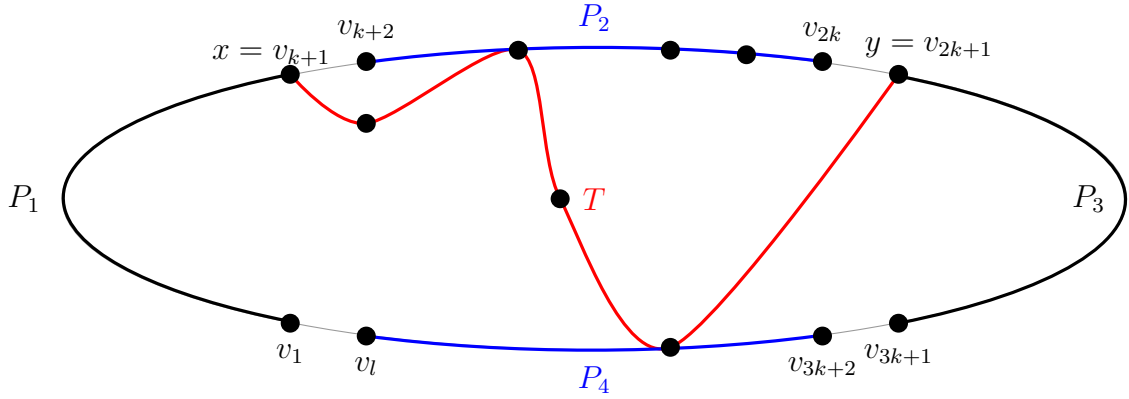

Now we complete the proof of Theorem~\ref{Main}.     By Claim~\ref{base}, $H$ has a maximum cycle $C_{1}$ with $|C_{1}|\ge mn^{1/3}$. By Claim~\ref{nottr}(i), $C_{1}$ is a 
 transversal. Let $C\subseteq G$ be a minimum cycle subject to
  $C$ being a transversal. By \eqref{Contra},  $|C|> 4n^{2/3}$.  
  By Claim~\ref{small}, there is a cycle $C^{*}\subseteq G$ with 
  $2n^{2/3}<|C^{*}|<|C|$. By minimality, $|C^{*}|$ is not 
  a transversal, but by Claim~\ref{nottr}(ii) $C^{*}$, is a transversal, a contradiction.  
\end{proof}

\section{Pairwise-intersecting families}

In this section we prove Proposition~\ref{prop:I} and speculate on possible improvements.

\begin{proof}[Proof of Proposition~\ref{prop:I}]

Suppose $\mathcal L(M,G)$ is not pairwise intersecting; say $Q,R\in 
\mathcal L(M,G)$ are disjoint. Let $\mathcal T:=\{T_{i}:i\in [t]\}$ be a
 maximum $Q,R$-connector. Define a bipartite multigraph $\mathcal H$ 
 with $V(\mathcal H)=\bigcup_{e\in F}(Q_{e}\cup R_{e})$ and 
 $E(\mathcal H)=\mathcal T$, where the ends of the edge $T_{i}$ in 
 $\mathcal H$ are $Q_{e}$ and $R_{e'}$ when the ends of the 
 path $T_{i}$ in $G$ are in $Q_{e}$ and $R_{e'}$; if there is more than one option for the end of the edge $T_{i}$ (when this end is a branch vertex with degree at least three), then make an arbitrary choice. Let $X$ be the set of cut edges in $F$, $Y=F\smallsetminus X$; then $c=|X|$. Using Menger's Theorem, it suffices to show that $\|\mathcal H\|\le m^{2} -cm+\binom{c}{2}$.

  
 We first show that $\mathcal H$ is simple; thus $\|\mathcal H\|\le m^{2}$ 
 (this is essentially the argument in \cite{MML}). If not, then there are 
 $e,e'\in F$ and distinct $i,j\in [t]$ with $T_{i}, T_{j}\in E(Q_{e},R_{e'})$. 
 But then, by Lemma~\ref{conpath}, we can enlarge $Q_{e}$ or $R_{e'}$, 
 contradicting maximality. 
 
 Now suppose $e=xy\in X$, $M_{1}$ 
 \& $M_{2}$ are the two components of $M-e$ with $y\in V(M_{2})$ 
 and $e'\in E(M_{2})$. For $i\in [2]$, put
  $Q^{i}=\bigcup_{f\in E(M_{i})}Q_{f}$ and 
  $R^{i}=\bigcup_{f\in E(M_{i})}R_{f}$. Suppose 
  $Q_{e}R_{e'},R_{e}Q_{e'}\in E(\mathcal H)$. Then there are paths 
  $P=u\dots v, P'=u'\dots v'\in\mathcal T$ with
    $u\in P_{e}$, $v\in R_{e'}$, $u'\in R_{e}$ and $v'\in Q_{e'}$.
    As $M$ is connected, there are a $q^{y},v'$-path $\dot Q\subseteq Q$ and
     an $r^{y},v$-path $\dot R\subseteq R$. See Figure~\ref{fig:Prop11}.

 \begin{figure}[h!]
    \label{fig:Prop11}
    \center{
    \begin{tikzpicture}[scale = 1.5]
    
    
    
    

    \draw[fill,opacity=0.1,color=red] (0,1) ellipse [x radius = 2, y radius = 0.7];
    \node[color=blue] (Q1) at (6,1) {};
    \node at (6,.1)  {$Q^1$};
    \draw[fill,opacity=0.07, color=blue] (Q1) ellipse [x radius = 1.5, y radius = 0.7];

    \draw[fill,opacity=0.07,color=blue] (0,-1) ellipse [x radius = 2, y radius = 0.7];
    \node[color=red] (R1) at (6,-1){};
    \node at (6,-1.85) {$R^1$};
    \draw[fill,opacity=0.1, color=red] (R1) ellipse [x radius = 1.5, y radius = 0.7];
    
    \node[fill, shape=circle,inner sep=0, minimum size = .25 cm] (Qep) at (-1, 1) {};
    \node[fill, shape=circle,inner sep=0, minimum size = .25 cm] (Rep) at (-1, -1) {};
    
    \node[fill,label = right:$q^x$,shape = circle,inner sep = 0 cm, minimum size = .25 cm] (qx) at (4.5,1) {};
    \node[fill,label = 272:$q^y$,shape = circle,inner sep = 0 cm, minimum size = .25 cm] (qy) at (2,1) {};
    \node[fill,label = right:$r^x$,shape = circle,inner sep = 0 cm, minimum size = .25 cm] (rx) at (4.5,-1) {};
    \node[fill,label = 272:$r^y$,shape = circle,inner sep = 0 cm, minimum size = .25 cm] (ry) at (2,-1) {};
    
    \node[fill,label = below:$v'$,shape = circle,inner sep = 0 cm, minimum size = .18 cm] (vp) at (0,1){};
    \node[fill,label = below:$v$,shape = circle,inner sep = 0 cm, minimum size = .18 cm] (v) at (0,-1){};
    
    \node[fill,label=below:$u$,shape = circle,inner sep = 0 cm, minimum size = .18 cm] (u) at (3.3,1){};
    \node[fill,label = below:$u'$, shape = circle,inner sep = 0 cm, minimum size = .18 cm] (up) at (3.3,-1){};
    
    \node[label={:$Q_e$}] at (3.6,1){};
    \node[label={:$R_e$}] at (3.6,-1){};
    \node[label={$Q_{e'}$}] at (-0.5,1){};
    \node[label={$R_{e'}$}] at (-0.5,-1){};

    \node[fill, shape=circle,inner sep=0, minimum size = .25 cm] (qen) at (1,1) {};
    \node[fill, shape=circle,inner sep=0, minimum size = .25 cm] (ren) at (1,-1) {};

    \draw[ultra thick,color=red] (qy)--(u);
    \draw[ultra thick,color=red] (ry)--(up);
    \draw[ultra thick,red] (v)--(ren);
    \draw[thick, red] (ry)--(ren);
    \draw[thick,red] (up)--(rx);
    \draw[thick, red] (u)--(v);
    
    \draw[thick,color = blue] (qx)--(u);
    \draw[thick,color=blue] (qy)--(qen);
    \draw[thick,color=blue] (vp)--(up);
    \draw[color=blue,ultra thick,dashed] (qy)--(u);
    \draw[color=blue,ultra thick,dashed] (up)--(ry);
    \draw[ultra thick, blue] (vp)--(qen);

    \draw[thick] (Qep)--(vp);
    \draw[thick] (Rep)--(v);
    \draw[dashed, ultra thick] (vp)--(qen);
    \draw[dashed,ultra thick] (v)--(ren);

    \node[label={:$\dot{Q}$}] at (1.25, 1) {};
    \node[label={:$P$}] at (2.5,.01){};
    \node[label={:$\dot{R}$}] at (1.25, -1) {};
    \node[label={:$P'$}] at (2.5,-0.55){};
    \node at (0,.1){$Q^2$};
    \node at (0,-1.85){$R^2$};

    \end{tikzpicture}
    }
    \caption{Violating maximality of $Q,R$}

\end{figure}
     
     One of the $M$-subdivisions
    \[Q^{1}\cup q^{x}Q_{e}q^{y}\dot Qv'P'u'R_{e}r^{y}\cup R^{2}
    \textrm{ and }
    R^{1}\cup r^{x}R_{e}r^{y}\dot RvPuQ_{e}q^{y}\cup Q^{2}\]
    has size greater than $\mu$, contradicting maximality. 
 Thus at most one of $Q_{e}R_{e'},R_{e}Q_{e'}$ is an edge of $\mathcal H$. So
 \begin{align*}
 \|\mathcal H\|&\le |Y|^2+\frac{1}{2}(|X||Y|+|Y||X|+|X|(|X|-1))\\&=(m-c)^{2}+c(m-c)+\frac{c(c-1)}{2}=m^{2}-cm+\binom{c}{2}. \qedhere
 \end{align*}

\end{proof}

We have no reason to believe that Proposition~\ref{Sep} is close to optimal. It may be
interesting to investigate the case that $M$ is a tree. For instance, it is not
hard to see that $Q_{e}R_{e'}\notin E(\mathcal{H})$ for all $e,e'\in F$ with $e,e'$
incident to leaves in $M$. Thus if $M$ is a star and $G$ is connected,  then $\mathcal L(M,G)$ is pairwise intersecting. It is well known \cite[Exercise 1.27]{Die} that if $G$
is a tree and $\mathcal{L}:=\mathcal{L}(M,G)$ is pairwise 
intersecting, then $\tau(\mathcal{L})=1$. 

We also have no reason to believe that Theorem~\ref{Main} is close to optimal.
For the case $M=K_{2}$ (maximum paths) $\tau$ could be constant. One advantage
of considering the more general problem is that it may be easier to develop techniques
for proving lower bounds in this setting.

\begin{ack}
We thank a referee for suggesting that we extend our original  argument for longest path transversals to maximum subdivision transversals.
\end{ack}

\bibliographystyle{plain}
\bibliography{Tau+}

\end{document}